\documentclass[11pt]{article}

\usepackage[a4paper, total={5.8in, 8in}]{geometry}
\usepackage{tikz-cd}
\usepackage{tikz}
\usepackage{microtype}
\usepackage{lipsum}
\usepackage{graphicx}
\usepackage{float}
\usepackage{biblatex}
\usepackage[hidelinks]{hyperref}

\usepackage{xurl}
\usepackage{mathtools}
\usepackage{amssymb}
\usepackage{amsthm}
\usepackage{amsmath}
\usepackage{setspace}
\usepackage[utf8]{inputenc}
\usepackage[english]{babel}
\usepackage{framed}
\usepackage{caption}
\usepackage[hang,flushmargin]{footmisc}
\usepackage{xcolor}
\usepackage{musicography}
\usepackage{csquotes}
\usepackage{marginnote}
\hypersetup{colorlinks,linkcolor={blue},citecolor={green},urlcolor={black}}  
\hypersetup{breaklinks=true}

\newtheorem{theorem}{Theorem}
\newtheorem{conjecture}{Conjecture}

\newtheorem{corollary}{Corollary}

\newtheorem{lemma}{Lemma}
\newtheorem{proposition}{Proposition}
\newtheorem{remark}{Remark}

\makeatletter

\newcommand\ackname{Acknowledgements}
\if@titlepage
  \newenvironment{acknowledgements}{%
      \titlepage
      \null\vfil
      \@beginparpenalty\@lowpenalty
      \begin{center}%
        \bfseries \ackname
        \@endparpenalty\@M
      \end{center}}%
     {\par\vfil\null\endtitlepage}
\else

\fi

\title{Uniqueness in the local Donaldson-Scaduto conjecture}
\author{Gorapada Bera, Saman Habibi Esfahani, Yang Li}

\date{\today}

\begin{document}
\maketitle

\begin{abstract}
The local Donaldson-Scaduto conjecture predicts the existence and uniqueness of a special Lagrangian pair of pants with three asymptotically cylindrical ends in the Calabi-Yau 3-fold $X \times \mathbb{R}^2$, where $X$ is an ALE hyperk\"ahler 4-manifold of $A_2$ type. The existence of this special Lagrangian has previously been proved \cite{esfahani2024donaldson}. In this paper, we prove a uniqueness theorem, showing that no other special Lagrangian pair of pants satisfies this conjecture.
\end{abstract}

\section{Introduction}

Donaldson and Scaduto conjectured the existence and uniqueness of an associative pair of pants submanifold in the $G_2$-manifold $X^4 \times \mathbb{R}^3$, where $(X^4,\omega_1,\omega_2,\omega_3)$ is a hyperk\"ahler $K3$ surface \cite[Conjecture 1]{donaldson2020associative}. Let $\alpha_1, \alpha_2, \alpha_3$ be $(-2)$-classes in $H^2(X^4;\mathbb{Z})$, that is, $\alpha_i^2 = -2$ with respect to the intersection product, such that $\alpha_1 + \alpha_2 + \alpha_3 = 0$. Each $\alpha_i$ determines a non-zero vector $v_i:=(\omega_1\cdot \alpha_i,\ \omega_2\cdot \alpha_i,\ \omega_3\cdot \alpha_i) \in \mathbb{R}^3$, and hence a complex structure $J_i$ in the $S^2$-family of complex structures on $X^4$. Assume each $\alpha_i$ is \textit{irreducible}, i.e., not decomposable into $(-2)$-classes all of which determine the same complex structure $J_i$.
Then $\alpha_i$ is represented uniquely by a smooth embedded $J_i$-holomorphic sphere $\Sigma_i$ in $X^4$ \cite[Section 4]{donaldson2020associative}. Moreover, the three cylindrical submanifolds $P_i := \Sigma_i \times (\mathbb{R}^+ \cdot v_i) \subset X^4 \times \mathbb{R}^3$ are associative submanifolds.

\begin{conjecture}[Donaldson-Scaduto conjecture]
\label{Global-Donaldson-Scaduto-conjecture}
There is a unique associative submanifold $P$ in $X^4 \times \mathbb{R}^3$ with three ends asymptotic to cylinders $P_1$, $P_2$, and $P_3$.
\end{conjecture}

Since $\alpha_1 + \alpha_2 + \alpha_3 = 0$, the vectors $v_i$ lie in a plane in $\mathbb R^3$, so, without loss of generality, we can assume that they are contained in $\mathbb{R}^2 \times \{0\}$. Consequently, the cylinders $P_i$ are special Lagrangians in $X^4 \times \mathbb{R}^2  \times \{0\} \subset X^4 \times \mathbb{R}^3$.
Thus, by the maximum principle (Lemma \ref{max-p}), an associative submanifold $P$ asymptotic to $P_i$ is a special Lagrangian submanifold contained in the Calabi-Yau 3-fold $X^4 \times \mathbb{R}^2$. Therefore, both the existence and uniqueness problems reduce to questions about special Lagrangian submanifolds.

The local version of the Donaldson-Scaduto conjecture predicts the existence and uniqueness of a special Lagrangian submanifold $L$ in the Calabi-Yau 3-fold $X \times \mathbb{R}^2$, where $X$ is an $A_2$-type ALE hyperk\"ahler 4-manifold with three holomorphic spheres $\Sigma_1, \Sigma_2, \Sigma_3$. $L$ has three ends, each asymptotic to $L_i := \Sigma_i \times (\mathbb{R}^+ \cdot \tilde{v}_i)$, where $\tilde{v}_i$ is the 90-degree rotation of $v_i$ in $\mathbb{R}^2$, ensuring the phase $\theta = 0$. The connection between the local and the global version is that when the $K3$ surface is the small desingularization of an orbifold with local $A_2$-singularity, the local version captures the metric behavior near the desingularization region.

The existence of a special Lagrangian $L_0 \subset X \times \mathbb{R}^2$ satisfying the local Donaldson-Scaduto conjecture is proved in \cite{esfahani2024donaldson}. The method also generalizes to the case where $X$ is an $A_{n-1}$-type ALE or ALF gravitational instanton with $n \geq 3$, given by the Gibbons-Hawking ansatz, where the monopole points $p_1, \ldots, p_n$ form a convex polygon arranged counterclockwise in a plane in $\mathbb{R}^3$. In this case, there are $n$ holomorphic spheres $\Sigma_1, \ldots, \Sigma_n$ in $X$, projecting to the edges of the convex polygon.

\begin{proposition}[Generalized local Donaldson-Scaduto conjecture: existence \cite{esfahani2024donaldson}]\label{Donaldson-Scaduto-conjecture} There is an $(n-1)$-dimensional family of $U(1)$-invariant special Lagrangian submanifolds in the Calabi-Yau 3-fold $X \times \mathbb{R}^2$, each homeomorphic to an $n$-holed 3-sphere and with $n$ asymptotically cylindrical ends, modeled on the product of $\Sigma_i$ and $\{y \in \mathbb{R}^2 \; | \; (p_{i+1} - p_i) \cdot y = c_i \}$, where the parameters $\{c_i\}_{i=1}^n$ satisfy one constraint $\sum_{i=1}^n c_i = 0$.
\end{proposition}

\begin{figure}[H]
\includegraphics[width=12cm]{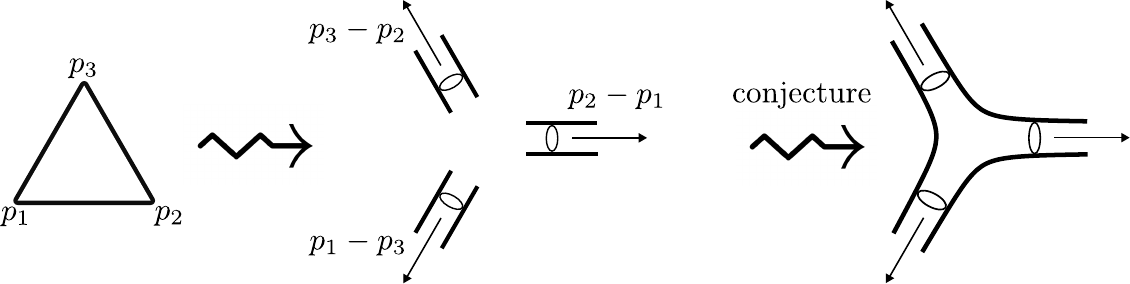}
\centering
  \caption{Local Donaldson-Scaduto conjecture $(n=3)$.}
  \label{Donaldson-Scaduto-figure}
\end{figure}

The $(n-1)$ parameters geometrically correspond to the translation of the $n$ model cylinderical ends, subject to the constraint $\sum_{i=1}^n c_i = 0$, which comes from the vanishing of the integral of $\text{Im}(\Omega)$ over the special Lagrangians. Two of these parameters correspond to global translations of special Lagrangians along the $\mathbb{R}^2$ directions, while the remaining $(n-3)$ parameters result in geometrically distinct special Lagrangian submanifolds. Moreover, these special Lagrangians are rigid after fixing the asymptotics, as shown in \cite{MR4495257}.

Here, we prove a uniqueness counterpart to Proposition \ref{Donaldson-Scaduto-conjecture}:

\begin{theorem}\label{main}
Let $L \subset X \times \mathbb{R}^2$ be a special Lagrangian submanifold homeomorphic to an $n$-holed 3-sphere with $n$ asymptotically cylindrical ends modeled on the product of $\Sigma_i$ and $\{y \in \mathbb{R}^2 \; | \; (p_{i+1} - p_i) \cdot y = c_i \}$. Then, $L$ is a member of the $(n-1)$-dimensional family of $U(1)$-invariant special Lagrangians constructed in \cite{esfahani2024donaldson}. 
\end{theorem}

Letting $n=3$, and by the maximum principle (see Lemma \ref{max-p}), we get the uniqueness:

\begin{corollary} [Local Donaldson-Scaduto conjecture: uniqueness] \label{uniqueness}
    Let $P$ be an associative pair of pants submanifold in the $G_2$-manifold $X \times \mathbb{R}^3$, with three ends asymptotic to the half-cylinders $\Sigma_i \times (\mathbb{R}^+ \cdot v_i)$, where $i \in \{1,2,3\}$. Then $P$ coincides with the associative submanifold constructed in \cite{esfahani2024donaldson}.
\end{corollary}

In this writing, we mainly focus on the case where $L$ is homeomorphic to an $n$-holed 3-sphere. In Appendix \ref{topology}, we show that if $L$ is any 3-manifold special Lagrangian satisfying the generalized local Donaldson-Scaduto conjecture, it is either an $n$-holed 3-sphere or an $n$-holed connected sum of finitely many $S^2 \times S^1$. From the Donaldson-Scaduto conjecture \ref{Global-Donaldson-Scaduto-conjecture}, it is expected that the latter possibility is not realized.

\begin{remark}
\emph{The special Lagrangians (resp. associatives) in the local Donaldson-Scaduto conjecture of Proposition \ref{Donaldson-Scaduto-conjecture} are expected to serve as the building blocks of the gluing constructions of special Lagrangians (resp. associatives) in the Calabi-Yau 3-folds (resp. $G_2$-manifolds) with $A_{n-1}$-type ALE or ALF Lefschetz fibration.} 
\end{remark}

\begin{remark}\emph{
     The Donaldson-Scaduto conjecture \ref{Donaldson-Scaduto-conjecture} remains unresolved for arbitrary hyperk\"ahler $K3$ surfaces at present.  However, the results on the local version of existence in Proposition \ref{Donaldson-Scaduto-conjecture} and of uniqueness in Corollary \ref{uniqueness} may serve as a first step toward proving the conjecture for a neighborhood of desingularizations or resolutions of Kummer surfaces with an $A_2$-singularity.}
\end{remark}

\begin{remark}\emph{
   Donaldson-Scaduto suggests that special Lagrangians (resp. associatives) in Calabi-Yau 3-folds (resp. $G_2$-manifolds) with Lefschetz fibration (resp. Kovalev--Lefschetz fibration) in the adiabatic limit may correspond to certain gradient graphs in the base manifold \cite[Section 4-5]{donaldson2020associative}. The existence and uniqueness in the Donaldson-Scaduto conjecture is the first step towards understanding such special Lagrangians in the vertex region of the gradient graphs.
   }
\end{remark}

\

\noindent
\textbf{Organization.} We focus on proving Theorem \ref{uniqueness}. We start by observing that $L$ must be $U(1)$-invariant, and its projection to the Gibbons-Hawking coordinates $(u_1,u_2)$ lands inside the interior of the convex polygon union with the vertices. The key step is to show that over the interior of the polygon, $L/U(1)$ can be expressed as a smooth graph; this is achieved by a combination of PDE and topological considerations. The uniqueness then boils down to the uniqueness of the solution for the Dirichlet problem of a real Monge-Amp\`ere equation. The Appendix \ref{topology} examines the possibility of $L$ having a different topology.

\

\noindent
\textbf{Acknowledgement.} We thank Professor Simon Donaldson for helpful discussions and the referees for their useful comments on the previous draft.

\section{Preliminaries}\label{Preliminaries}

We review the hyperk\"ahler structure on the $U(1)$-invariant gravitational instanton $X$, the Calabi-Yau structure on $Z = X \times \mathbb{R}^2$, and the $G_2$-structure on $M = X \times \mathbb{R}^3$, as well as asymptotically cylindrical special Lagrangians and associative submanifolds.

\

\noindent
\textbf{Hyperk\"ahler structure.} 
Let $X$ be a complete non-compact $U(1)$-invariant hyperk\"ahler 4-manifold constructed using the Gibbons-Hawking ansatz.

Let $u_1, u_2, u_3$ be coordinates on $\mathbb{R}^3$. For $n \geq 3$, let $p_1, p_2, \ldots, p_n$ denote $n$ distinct points in $\mathbb{R}^3$ that lie on the plane $\mathbb{R}^2 \times \{0\}$ and are arranged in counterclockwise order to form a convex polygon. Consider a principal $U(1)$-bundle $\pi: X^{\circ} \to \mathbb{R}^3 \setminus \{p_1, p_2, \ldots, p_n\}$ with Chern class equal to 1 around each point $p_i$. Let $V: \mathbb{R}^3 \setminus \{p_1, p_2, \ldots, p_n\} \to \mathbb{R}$ be the positive harmonic function
\[
V(u) = A + \sum_{i=1}^n \frac{1}{2 |u - p_i|},
\]
where $A \geq 0$ is a constant. 

Let $\theta$ be a $U(1)$-connection on $X^{\circ}$ with the curvature 2-form $d\theta=-*dV$. The Gibbons-Hawking ansatz provides a hyperk\"ahler structure on $X^{\circ}$ with symplectic forms
\[
\omega_i = \theta \wedge du_i + V \, du_{i+1} \wedge du_{i+2},
\]
with indices understood in the cyclic sense, and the metric $g = V^{-1} \theta^2 + V \sum_{i=1}^3 du_i^2$.

The coordinates $u_1, u_2, u_3$ are the moment maps for the $U(1)$-action for the symplectic forms $\omega_1, \omega_2, \omega_3$, respectively. The manifold $X$ is obtained by adding a point $q_i$ above each point $p_i$, and the hyperk\"ahler structure extends smoothly to $X$ with complex structures $I_1, I_2, I_3$. For each $(a_1, a_2, a_3) \in S^2 \subset \mathbb{R}^3$, we get a complex structure $\sum_{i=1}^3 a_i I_i$ on $X$. When $A=0$, $X$ is an $A_{n-1}$-type ALE space, and for $A>0$, it is an $A_{n-1}$-type ALF space.

Let $\Sigma_i = \pi^{-1}[p_i, p_{i+1}]$ be the preimage of the line segment $[p_i, p_{i+1}]$, where $p_{n+1} = p_1$. Each $\Sigma_i$ is a 2-sphere, which is holomorphic with respect to the complex structure associated with the vector $v_i = (p_{i+1} - p_{i})/|p_{i+1} - p_{i}| \in S^2 \cap (\mathbb{R}^2_{(u_1, u_2)} \times \{0\}) \subset \mathbb{R}^3_{(u_1, u_2, u_3)}$.

\

\noindent 
\textbf{Calabi-Yau structure.} Let $Z = X \times \mathbb{R}^2_{(y_1,y_2)}$. The $6$-dimensional manifold $Z$ can be equipped with the Calabi-Yau structure
\[
g_Z = g_X + g_{\mathbb{R}^2}, \quad \omega = \omega_3 + dy_2 \wedge dy_1,\quad \Omega = (\omega_1 + i \omega_2) \wedge (dy_2 + i dy_1),
\]
where $y_1, y_2$ denote the coordinates on $\mathbb{R}^2$. With our convention $\omega^3 = \frac{3\sqrt{-1}}{4} \Omega \wedge \overline{\Omega}$.

The $U(1)$-action on the Gibbons-Hawking space $X$ extends to a $U(1)$-action on $Z$ by \[e^{i t} \cdot (q,y) \to (e^{it} \cdot q, y), \; \text{ for all }  \; q \in X \; \text{ and } \; y \in \mathbb{R}^2.\] This $U(1)$-action is Hamiltonian with the moment map $u_3: Z \to \mathbb{R}$.

Let $\tilde{v}_i= Rv_i$, where $R: \mathbb{R}^2_{(u_1,u_2)}\to \mathbb{R}^2_{(y_1,y_2)}$ is the linear transformation given by the clockwise 90-degree rotation, 
\[R( \partial_{u_1})= -\partial_{y_2},\quad R( \partial_{u_2})= \partial_{y_1}.\]
Let $L_{i} $ be $ \Sigma_{i} \times (\mathbb{R}^+ \cdot \tilde{v}_{i}) \subset X \times \mathbb{R}^2_{(y_1,y_2)}$ translated along some vector $\tau_i$ in  $\mathbb{R}^2_{(y_1,y_2)}$, so that $L_i$ is contained inside
\[
\Sigma_i \times \{ y \in \mathbb{R}^2 \; | \; (p_{i+1}-p_i)\cdot y=c_i \}\subset X\times \mathbb{R}^2,\quad i \in \{1,2,\ldots, n\},
\]
where $c_i:=(p_{i+1}-p_i)\cdot \tau_i$. A direct computation shows that $L_i$ are $U(1)$-invariant special Lagrangians with phase $\theta = 0$ in $Z$.

\

\noindent
\textbf{$G_2$-structure.} Let $M = X \times \mathbb{R}^3_{(y_1, y_2, y_3)}$. The 7-dimensional manifold $M$ can be equipped with a torsion-free $G_2$-structure 
\begin{align*}
    \phi = dy_1 \wedge dy_2 \wedge dy_3 
      + \sum_{i=1}^3 dy_i \wedge du_i \wedge \theta - \sum_{i=1}^3 V dy_i \wedge du_{i+1} \wedge du_{i+2},
\end{align*}
with cyclic indices. The associated $G_2$-metric is $g_M = g_X + g_{\mathbb{R}^3}$.

Let $P_i = \Sigma_i \times (\mathbb{R}^+ \cdot v_i) \subset X \times \mathbb{R}^3$ translated along some vector in $\mathbb{R}^2_{(y_1,y_2)} \times \{0\}$, where $v_i$ is a vector in $\mathbb{R}^2 \times \{0\} \subset \mathbb{R}^3$. The cylinders $P_i$ are $U(1)$-invariant associative submanifolds in $M$. With the identification $M \cong Z \times \mathbb{R}_{y_3}$, we have $\phi = -dy_3 \wedge \omega - \text{Im}(\Omega)$. For a submanifold $L$ in $X \times \mathbb{R}^2$, we define $\tilde{L} = \{(q, y) \mid (q, R^{-1} y) \in L\}$. The submanifold $P = L \times \{ 0\}$ in $M$ is an associative if and only if $L$ is a special Lagrangian in $X \times \mathbb{R}^2$ with phase $\pi/2$, or equivalently, $\tilde L$ is a special Lagrangian in $X \times \mathbb{R}^2$ with phase $0$.  

\

\noindent
\textbf{Asymptotically cylindrical associatives and special Lagrangian submanifolds.} Let $j_{i}: U_{i}\subset NP_i\to M $ be a tubular neighborhood map for the associative cylinder $P_i$, where $NP_i$ denotes the normal bundle of $P_i$ in $M$. A non-compact associative submanifold $P$ in $M=X\times \mathbb R^3$ is said to be asymptotically cylindrical with asymptotic ends $P_i$, for $i = 1, \ldots, n$, if there exist $R>0$, a compact subset $K_P\subset P$, and normal vector fields $\nu_i$ on each $\Sigma_{i} \times \left ((R,\infty)\cdot v_i \right)$ with $\bigcup_ij_{i}(\operatorname{graph} \nu_i)=P\setminus K_P$ such that 
\begin{equation}\label{eq-def-acyl-asso}
  |{\nu_i}|=o(1)\quad \text{and} \quad  |{\nabla\nu_i}|=o(1) \quad \text{as} \quad t\to +\infty.
\end{equation}

The Corollary \ref{uniqueness} follows from the Theorem \ref{main} and the following lemma.

\begin{lemma}\label{max-p} Let $P$ be an asymptotically cylindrical associative submanifold in $M$ with asymptotic ends $P_i$ $\subset Z \times \{0\}$, where $i=1,\dots,n$. Then $P$ is of the form $L \times \{0\}$, where $L$ is an asymptotically cylindrical special Lagrangian submanifold in $Z$ with phase $\pi/2$. 
\end{lemma}

\begin{proof} 
Let $y_3: P \to \mathbb{R}$ be the restriction of the $y_3$-coordinate function to $P$. Since $P$ is a minimal submanifold in $X \times \mathbb{R}^3$, it follows that the restriction of $y_3$ to $P$ is a harmonic function. We see that $y_3$ agrees with $\langle \nu_i, \partial_{y_3} \rangle$ on the asymptotic ends $P_i$, which vanishes at infinity. Therefore, by the maximum principle, we have $y_3 = 0$ on $P$, and thus $P$ is of the form $L \times \{0\} \subset Z \times \{0\}$.
\end{proof}

Let $L$ be an asymptotically cylindrical special Lagrangian in $Z$ with asymptotic ends $L_i$.
The following lemma upgrades the decay in Equation (\ref{eq-def-acyl-asso}) to exponential decay in all derivatives. Since the proof is standard (see \cite[Thm 6.8]{JoyceSLregularity2004} for a similar argument), we only provide a sketch.
\begin{lemma}The normal vector fields $\nu_i$ over the end of $L_i$ from Equation (\ref{eq-def-acyl-asso}) satisfy the following exponential decay estimates: there are constants $\mu_i>0$ such that for each $k\geq 0$,
\begin{equation}\label{eq exponential decay} |{\nabla^k\nu_i}|=O(e^{-\mu_i t}) \quad \text{as} \quad t\to +\infty.
\end{equation}
\end{lemma}
\begin{proof}[{Sketch of the proof}]
Since each $\Sigma_i$ is a special Lagrangian sphere in $X$ with some phase, it does not admit any non-trivial infinitesimal special Lagrangian deformations. Therefore by \cite[Theorem 1(i)]{Adams1988}, there are constants $\mu_i>0$ such that
\begin{equation}
\label{eq adams}
|{\nu_i}|=O(e^{-\mu_i t}) \quad \text{and} \quad |{\nabla\nu_i}|=O(e^{-\mu_i t}) \quad \text{as} \quad t\to +\infty.
\end{equation}
It is worth noting that the hypothesis in the statement of that theorem concerns the kernel of the deformation operator for minimal submanifolds, which differs from our setting. However, the proof presented there \cite[pp. 234–235]{Adams1988} employs a `blowing up' method, can be adapted to our context as follows. If the Equation (\ref{eq adams}) does not hold then for $T\gg 1$, a subsequence of 
\[\hat \nu^k_i(t):=\varepsilon_k^{-1} \nu_i(t+kT), \quad \text{with} \quad \varepsilon_k:=\sup_{[kT, (k+1)T]}|\nu_i|>0\] converges locally in $C^2$ as $k\to +\infty$ over the end $L_i$ to a normal vector field $\hat \nu_i$ that satisfies the linearization equation at $L_i$ for special Lagrangians in $Z$:
\[\big(D_{L_i}:=\tfrac d{dt}+D_{\Sigma_i}\big)\hat \nu_i=0,\]
where $D_{\Sigma_i}$ is the deformation operator at $\Sigma_i$ for special Lagrangians in $X$. 
Moreover, the fact that $\ker D_{\Sigma_i}=0$ implies that $\hat \nu_i=\sum_{\lambda>0} e^{-\lambda t} \phi_{i,\lambda}$,
where $\phi_{i,\lambda}$ is a eigensection of $D_{\Sigma_i}$ corresponding to the eigenvalue $\lambda$. This would then lead to a contradiction as explained there. 

Moreover, as noted in \cite[Eq.~(2.2)]{Adams1988}, we have $|\nabla^2 \nu_i| = o(1)$ as $t \to +\infty$. By adapting the `blowing up' method used in the proof of \cite[Part~II, Theorem~6.6]{Simon1985} to our setting, it follows that
\begin{equation*}
|\nabla^2 \nu_i| = O(e^{-\mu_i t}) \quad \text{as} \quad t \to +\infty.
\end{equation*}

The exponential decay of all higher derivatives then follows from `Schauder regularity estimates in weighted spaces' \cite[Theorem~4.12]{marshal2002deformations}, since $\nu_i$ solves a quasi-linear equation of the form:
\[
a_1(\nu_i,\nabla \nu_i)\, \nabla^2 \nu_i + a_0(\nu_i,\nabla \nu_i) = 0,
\]
which arises by applying the operator $D_{L_i}$ to the special Lagrangian equation. A similar argument is presented in the asymptotically conical setting in \cite[Theorem~6.43]{marshal2002deformations}.
\end{proof}

\section{Proof of Theorem \ref{uniqueness}}\label{proof}

Let $L$ be a smooth special Lagrangian submanifold that satisfies the asymptotic condition of the generalized local Donaldson-Scaduto conjecture. In this section, we prove that $L$ is $U(1)$-invariant (Lemma \ref{U(1)-invariant}). Furthermore, we show that when $L$ is homeomorphic to an $n$-holed 3-sphere, the quotient $L/U(1)$ satisfies a graphicality condition (Lemma \ref{Lemma_submersion}). This result is then applied to prove Theorem \ref{main}.

\begin{lemma}\label{U(1)-invariant}
    The special Lagrangian $L$ is invariant under the $U(1)$-action.
\end{lemma}

\begin{proof}
    The $U(1)$-action on $L$ generates a 1-parameter family of special Lagrangians $e^{i \theta} \cdot L$ with $\theta \in [0, 2 \pi)$. The Hamiltonian function generating this 1-parameter deformation family of $L$ is the moment map coordinate $u_3 : Z \to \mathbb{R}$; i.e., $du_3 = \iota_{\partial_{\theta}} \omega$, where $\partial_{\theta}$ is the vector field associated to the $U(1)$-action. This shows $u_3$ is a harmonic function on $L$
    \cite[Proposition 6.46]{marshal2002deformations}.
    Since $L$ is asymptotic to cylinders $L_i$, we know $u_3$ decays to zero along all the ends of $L$, so by the maximum principle $u_3 = 0$ on $L$.  Thus the $U(1)$ vector field $\partial_{\theta}$ is tangent to $L$, hence, $L$ is $U(1)$-invariant.
\end{proof}

The $U(1)$-moment map associated to the symplectic form $\omega$ on $Z$ is $u_3$, which must be constant on the $U(1)$-invariant Lagrangian $L$, and this constant is zero since it is zero on the asymptotic cylindrical models. The dimensionally reduced special Lagrangian in the symplectic quotient is
\begin{align*}
    L_{\text{red}} := L/U(1) \subset Z_{\text{red}} := u_3^{-1}(0)/U(1) = \mathbb{R}^2_{(u_1, u_2)} \times \mathbb{R}^2_{(y_1,y_2)}.
\end{align*}   
An essential property of this $U(1)$-action on $Z$, and consequently on $L$, is that each orbit of this action is either free or a fixed point. Let
    \begin{align*}
        L_{\text{free}} := L \cap \left( (X \setminus \{q_1, \ldots, q_n\}) \times \mathbb{R}^2 \right), \quad \quad
        L_{\text{fix}} := L \cap (\{q_1, \ldots, q_n\} \times \mathbb{R}^2).
    \end{align*}
    We have the disjoint union $L = L_{\text{free}} \cup L_{\text{fix}}$.

The $U(1)$-invariance simplifies the special Lagrangian conditions for $L \subset X \times \mathbb{R}^2$ to the following equations:  
\begin{align}\label{SLeq}
        V du_1 \wedge du_2 = dy_1 \wedge dy_2, \quad \quad \text{and}  \quad \quad du_1 \wedge dy_1 + du_2 \wedge dy_2 = 0.
    \end{align}

\begin{lemma}
\label{Lemma_manifold with boundary}
{$L_{\text{red}}$ is a smooth 2-manifold with boundary, where $\partial L_{\text{red}} = L_{\text{fix}}/U(1) \cong L_{\text{fix}}$. Furthermore, it has $n$ boundary components that are homeomorphic to $\mathbb{R}$, connecting the $L_i$ end to the $L_{i+1}$ end for $i = 1, 2, \dots, n$. Additionally, if $L$ is simply connected, then $L_{\text{red}}$ is simply connected and has no genus and no circle boundary components.} 
\end{lemma}

We see a schematic presentation of $L_{\text{red}}$ in $Z_{\text{red}}$ in Figure \ref{LDSC} in the case $n=3$.

\begin{figure}[H]  
\includegraphics[width=9cm]{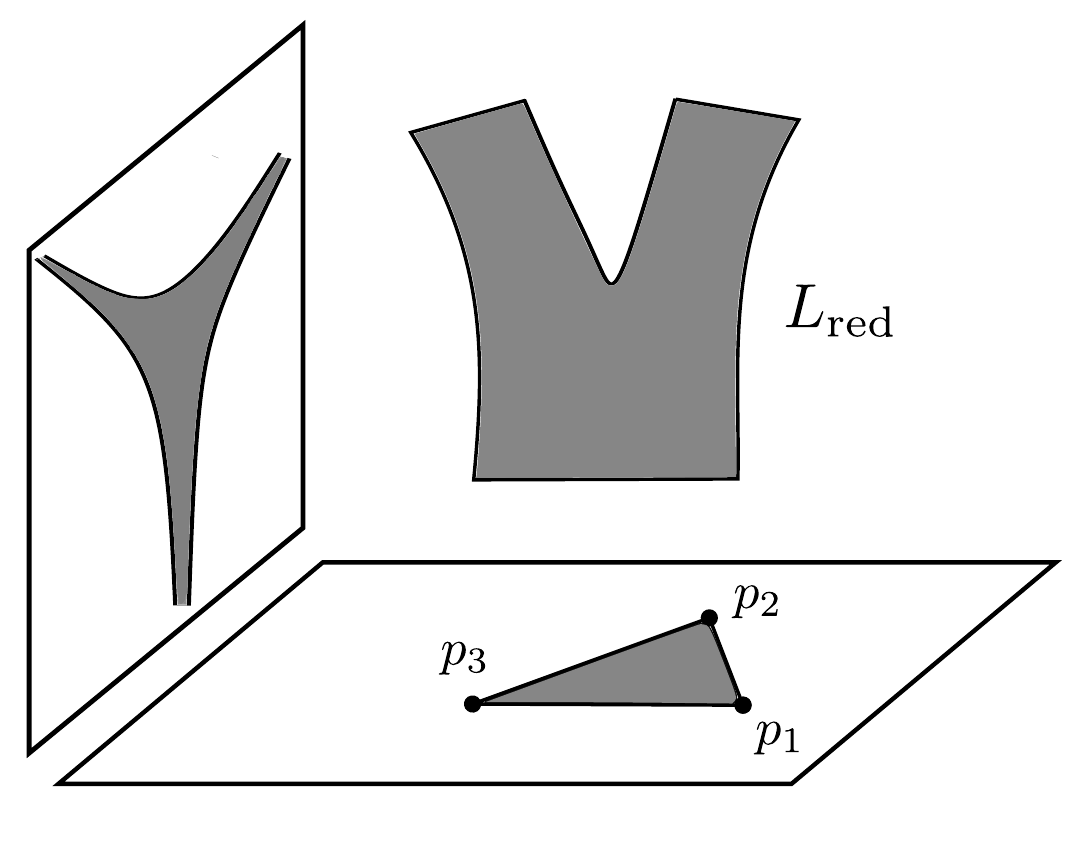}
\centering
  \caption{$L_{\text{red}}$ in $Z_{\text{red}}$ in the case $n=3$.}\label{LDSC}
\end{figure}

\begin{proof}

\textbf{Step 1} ($L_{\text{red}}$ is a smooth $2$-manifold with $\partial L_{\text{red}} = L_{\text{fix}}/U(1) \cong L_{\text{fix}}$).  Let $O_x$ be the $U(1)$-orbit of a point $x\in L$, and let $\Gamma_x$ be the stabilizer subgroup of $U(1)$ at $x$. By the slice theorem \cite[Theorem 2.4.1]{Duistermaat_2000}, there exists a $\Gamma_x$-invariant open neighborhood $V_x$ of $x$ in $N_xO_x\subset T_xL$ and a $U(1)$-invariant open neighborhood $U_x$ of $x$ in $L$ such that the exponential map
\[\operatorname{exp}^L_x:U(1)\times_{\Gamma_x} V_x\to U_x\]
is a $U(1)$-equivariant diffeomorphism. This implies that $L_{\text{free}}$ is an open smooth submanifold of $L$, and $L_{\text{free}}/U(1)$ is a smooth submanifold of $Z_{\text{red}}$ of dimension $2$. On the other hand, if $x\in L_{\text{fix}}$, then $O_x=\{x\}$, $V_x$ is $U(1)$-invariant, $U(1)\times_{\Gamma_x} V_x=V_x$, and $\operatorname{exp}^L_x: V_x\to U_x$ is a $U(1)$-equivariant diffeomorphism. Since the $U(1)$ action preserves $Re(\Omega)$, its action on $T_xL$ defines an orientation-preserving real $3$-dimensional representation. Consequently, it decomposes as $T_xL\cong \mathbb R\oplus \mathbb C$, where $U(1)$ acts trivially on $\mathbb R$ and acts on $\mathbb C$ by standard complex multiplication. Crucially, the latter action cannot be trivial, since $L_{\text{fix}}$  is locally planar, and the weight must be one, as the $U(1)$-action has no non-trivial finite isotropy subgroups. This implies that $L_{\text{fix}}$ is a $1$-dimensional smooth submanifold of $L$.

Next, we show that $L_{\text{red}}$ is a smooth manifold with boundary. Suppose $x=(q_i,y)\in L_{\text{fix}}$. Let $I_x= \{q_i\}\times I_y$ be a small open neighborhood of $x$ in $L_{\text{fix}}$. Then, there is a $U(1)$-invariant open tubular neighborhood $D_{q_i}\times I_y$ in $NI_x$ of $I_x$, with $D_{q_i}\subset T_{q_i}X$ such that $$\operatorname{exp}^L_{I_x}: D_{q_i}\times I_y\hookrightarrow L,$$ 
is a $U(1)$-equivariant tubular neighborhood map. With the quotient map $h_0:\mathbb C\to [0,\infty)$ given by $z\mapsto |z|^2$, the above tubular neighborhood maps descend to maps
$$[0,\epsilon_i)\times I_y\to L/U(1).$$
These maps induce a smooth manifold with boundary structure on $L/U(1)$ and $\partial L_{\text{red}}$ is $L_{\text{fix}}/U(1)\cong L_{\text{fix}}$.  

\textbf{Step 2} (Structure of $L_{\text{fix}}$). {Along each asymptotic cylinder modeled on $L_i$, there are two components of $L_{\text{fix}}$, since the $S^2$ cross section has two $U(1)$ fixed points, lying above $p_i$ and $p_{i+1}$. Since $L_{\text{fix}}$ is a one-dimensional submanifold, the components in the ends $L_i$ and $L_{i-1}$ that are projecting to the same $p_i$ must be part of one boundary component of $L_{\text{fix}}$ homeomorphic to $\mathbb R$. In this way, $L_{\text{fix}}$ has $n$ boundary components that are homeomorphic to $\mathbb{R}$, connecting the $L_i$ end to the $L_{i+1}$ end for $i = 1, 2, \dots, n$.}  

Additionally, if $L$ is simply connected, then $L_{\text{red}}$ is also simply connected. This follows from the following path lifting property for the quotient map $L \rightarrow L_{\text{red}}$: any given loop in $L_{\text{red}}$ can be lifted to a path connecting some point $x$ and $e^{i\theta}x$, which is further homotopic to the standard path $\{e^{it\theta}x:t\in [0,1]\}$ because $L$ is simply connected. The latter path projects to a constant, so the loop we started with in $L_{\text{red}}$ is contractible.

We now prove that $L_{\text{red}}$ has no genus or circle boundary components. This can be seen by compactfying each end of $L_{\text{red}}$ separately at infinity so that the compactification $\overline{L_{\text{red}}}$ is a compact simply connected surface with boundary. This has no genus and only one circle boundary component, which also contains all the above $n$ boundary components of $L_{\text{red}}$. This completes the proof.
\end{proof}

We now proceed to show that $L_{\text{red}}$ is the graph of a smooth map over the {interior of the convex polygon}. First, we control the image of $L_{\text{red}}$ under the projection map {\(\pi_u:Z_\text{red}\to \mathbb R^2_{(u_1,u_2)}\).}

\begin{lemma}\label{lem-proj-to-polygon}
    The projection of $L$ to the $(u_1, u_2)$ plane is contained in the union of the interior of the polygon domain and its vertices. 
\end{lemma}

\begin{proof}
        The interior of the polygon domain is the intersection of open half-planes, so it suffices to prove that $L$ projects inside any of the half-planes if it is not projecting to a vertex $p_i$. By a rotation of coordinates, without loss of generality, we can assume that the half-plane is $\{u_1 \geq 0\}$. We suppose the contrary that either (i) $\inf_{L} \{u_1\} = a < 0$  or (ii) there is a point $p \in L$ that projects to the open edge with $\inf_L \{u_1\} = 0$. If $\inf_L {u_1} \leq 0$, then, as the asymptotic ends of $L$ project to arbitrarily small neighborhoods of the edges of the polygon, in both cases (i) and (ii), $\inf_L {u_1}$ is attained at some interior point $p \in L$. At $p$, the minimum condition gives ${du_1}_{|_p} = 0$. The volume form on the special Lagrangian is $Re(\Omega) = \theta \wedge (du_1 \wedge dy_2 - du_2 \wedge dy_1) > 0$.

        Therefore, at the point $p$, we must have $du_2 \wedge dy_1 \neq 0$. By the implicit function theorem, in a small neighborhood of $p$, $L/U(1)$ can be locally expressed as the graph of $(u_1, y_2)$ as a function of $(u_2, y_1)$. The $U(1)$-invariant special Lagrangian Equations (\ref{SLeq}) for $(u_1, y_2)$ in terms of $(u_2, y_1)$ take the following form:
    \begin{equation}\label{eqn:ellipticsystem}
        V \frac{\partial u_1}{\partial y_1}  = \frac{\partial y_2}{\partial u_2}, 
        \quad \quad \text{and}  \quad \quad
        \frac{\partial u_1}{\partial u_2}  = -\frac{\partial y_2}{\partial y_1}.
    \end{equation}
    Thus by differentiation, we get 
    \begin{equation*}
        \frac{\partial^2 u_1}{\partial u_2^2}
        =
        - \frac{\partial^2 y_2}{\partial u_2 \partial y_1} 
        = - \frac \partial {\partial y_1} \Big(V \frac{\partial u_1} {\partial y_1} \Big),
    \end{equation*}
    namely, we have a second-order elliptic PDE for $u_1$,
    \begin{equation}\label{eq elliptic pde for max principle}
        V \frac{\partial^2 u_1}{\partial y_1^2} + \frac{\partial^2 u_1}{\partial u_2^2} + \frac{\partial V}{\partial y_1 } \frac{\partial u_1}{\partial y_1} = 0.
    \end{equation}
    By assumption, \( u_1 \) attains a local minimum at \( p \), and by the strong maximum principle, \( u_1 \) must be equal to the constant \( a \) in the local chart of \( L \). Since \( u_1 \) is real analytic, this implies that \( u_1 = a \) on \( L \), which contradicts the asymptotic behavior at infinity. 
\end{proof}

\begin{remark}\emph{
    The above proof can be compared to the standard fact that on minimal surfaces in $\mathbb{R}^n$, the coordinate functions $u_i$ on $\mathbb{R}^n$ satisfy the maximum principle since $u_i$ are harmonic functions on the minimal surface. In our case, the Gibbons-Hawking coordinates $u_1, u_2$ have non-trivial Hessian on the ambient space, so the harmonic function argument does not work; however, the maximum principle still holds by utilizing the full strength of the special Lagrangian condition.}
\end{remark}

\begin{remark}\emph{
    Note that the fixed points of $L$ project to the vertices of the polygon, and this does not contradict the maximum principle, as $L_{\text{red}}$ cannot be realized as a local graph at these points.}
\end{remark}

Lemma \ref{proper} proves a properness property for the projection $\pi_u$. We later use this to define the degree of the projection map, which is essential in the proof of graphicality.

\begin{lemma}[Properness of the projection away from vertices] \label{proper} The preimage of any compact subset $K$ of the open polygon is contained in a compact subset of the special Lagrangian. 
\end{lemma}

\begin{proof}
Any sequence of points on $L$ diverging to infinity lands in some end of $L$, so by the asymptotic cylindrical assumption, the projection lands in a small neighborhood of the corresponding edge of the polygon, which is eventually disjoint from $K$. 
\end{proof}

To prove the graphicality, at least near the open edges, we need a more precise description of the asymptotic behavior of the special Lagrangians.

\begin{lemma}[Asymptotic formulas of $L$]\label{thm_Asymp formula of SL} The normal vector fields $\nu_i$ over the end of $L_i$ from equation (\ref{eq-def-acyl-asso}) satisfy the following asymptotic formula: there are constants $\lambda_i>0$ and non-trivial translation invariant normal vector fields $\xi_i$ over $L_i$ such that for each $k\geq 0$,
\begin{equation}\label{eq asymp formula 1 of SL} 
{\nabla^k \big(\nu_i-e^{-\lambda_i t}\xi_{i}\big)}=o(e^{-\lambda_i t}), \quad \text{ as } \quad t \to \infty.
\end{equation}
\end{lemma}

\begin{proof} 
First we observe that $|\nu_i|$ cannot decay faster than exponentially along the $L_i$ end. After a rotation of the $u_1, u_2$ plane, we may assume that the open edge $(p_i,p_{i+1})$ lies on $\{ u_1=0\}$, the polygon lies on $\{u_1\geq 0\}$, and we locally express $(u_1, y_2)$ as functions of $(u_2, y_1)$ satisfying the elliptic system (\ref{eqn:ellipticsystem}). We apply the Harnack inequality to the elliptic equation (\ref{eq elliptic pde for max principle}) satisfied by the positively valued function $u_1$. Notably, the decay of $\partial_{y_1}u_1$, which follows from the asymptotically cylindrical condition along this end, is used here. The positivity of $u_1$ is a consequence of Lemma \ref{lem-proj-to-polygon}.  This shows that in the aymptotically cylindrical region, for $u_2$ in a fixed compact subset of the interval $(p_i, p_{i+1})$, we have the uniform equivalence
\[
C^{-1}u_1(u_2,y_1+1) \leq u_1(u_2,y_1) \leq C u_1(u_2, y_1+1)
\]
whence $|u_1| \geq C' e^{-Cy_1}$ for large $y_1$.

Now we argue for the leading order asymptote of $\nu_i$. 
We identify $NL_i\cong \Lambda^1(L_i)$ by the map $\nu \mapsto \iota_\nu\omega$. There is a tubular neighborhood $\mathcal U$ of $L_i$ and a nonlinear map
$\mathcal F:\mathcal U \subset \Gamma(N L_i)\cong \Omega^1(L_i)\to \Omega^1(L_i) \oplus \Omega^0(L_i)$
defined by \[\mathcal F(\iota_\nu\omega)=\Big(*\big(\exp_{L_i}\nu\big)^*\omega, *\big(\exp_{L_i}\nu\big)^* \operatorname{Im}\Omega\Big),\]
where $*$ denotes the Hodge star operator of $L$.

The elements in $\mathcal F^{-1}(0)$ correspond to special Lagrangians near $L_i$ under the tubular neighborhood map $\exp_{L_i}$. The linearization of $\mathcal F$ at $0$ is given by \cite[Theorem 3.6]{McLean1998}
\[d\mathcal F_{0}:\Omega^1(L_i)\to \Omega^1(L_i) \oplus \Omega^0(L_i), \quad \quad d\mathcal F_{0}=*d\oplus d^*.\] 
For any $\delta<0$, we define
\[\Omega^{j}_\delta(L_i):=\{\sigma\in \Omega^{j}(L_i): \nabla^k \sigma=O(e^{\delta t}),\ \forall k\geq 0, \ \text{as $t\to +\infty$}\}.\]
The restriction map $\mathcal F:\mathcal U\cap \Omega_\delta^1(L_i)\to \Omega_\delta^1(L_i) \oplus \Omega_\delta^0(L_i)$ can be extended to a map \[\tilde{\mathcal F}:\big(\mathcal U\cap \Omega_\delta^1(L_i)\big) \times \Omega^{0}_\delta(L_i) \to \Omega_\delta^1(L_i) \times \Omega_\delta^0(L_i),\]
defined by $\tilde{\mathcal F}(\sigma,f)=\tilde{\mathcal F}(\sigma)+(df,0)$. 
Thus 
$\tilde{\mathcal F}=\mathbb D_{L_i}+\mathcal Q$, 
\[\mathbb D_{L_i}:\Omega^1(L_i)\oplus \Omega^0(L_i)\to \Omega^1(L_i) \oplus \Omega^0(L_i), \quad \quad \mathbb D_{L_i}(\sigma,f)=(*d\sigma+df,d^*\sigma),\]
and the quadratic error term satisfies
\[|\mathcal Q (\sigma)-\mathcal Q (\sigma^\prime)|\lesssim \big(|\sigma-\sigma^\prime|+|\nabla (\sigma-\sigma^\prime)|\big)^2.\]
The operator 
$\mathbb D_{L_i}=\frac d{dt}+ \mathbb D_{\Sigma_i}$, where
$\mathbb D_{\Sigma_i}$ is a first order self-adjoint elliptic operator on $\Omega^1(\Sigma_i)\oplus \Omega^0(\Sigma_i)\oplus \Omega^0(\Sigma_i)$.

The normal vector fields $\nu_i$ have decay rate $\mu_i$ as in the equation (\ref{eq exponential decay}). 
Set $\sigma_i=\iota_{\nu_i}\omega$. For large enough $t$ along the asymptotic cylinder, $\mathcal F(\sigma_i)=0$, so
\[\mathbb D_{L_i}( \sigma_i,0)=O(e^{-2\mu_i t})\quad \text{and} \quad \mathbb D_{L_i}( \sigma_i,0)\in \Omega_{-2\mu_i}^1(L_i)\oplus \Omega_{-2\mu_i}^0(L_i).\]  By the discreteness of spectrum we may assume that $2\mu_i$ is not an eigenvalue of $\mathbb D_{\Sigma_i}$. Then there exist $\beta_i\in \Omega^{1}_{-2\mu_i}(L_i)$ and $f_i\in \Omega_{-2\mu_i}^0(L_i)$ such that $\mathbb D_{L_i}( \sigma_i,0)=\mathbb D_{L_i}( \beta_i,f_i)$ \cite[Section 3.1]{Donaldson2002}. Using the $L^2$-orthogonally decomposition of $(\sigma_i-\beta_i,-f_i)\in \operatorname{ker} \mathbb D_{L_i}$ we obtain that for large enough $t$,
\[ (\sigma_i-\beta_i, -f_i)=\sum_{\mu_i\leq \lambda\in \operatorname{spec} \mathbb D_{\Sigma_i} } e^{-\lambda t} \eta_{\lambda}, \]
where $(\eta_{\lambda},0)$ are $\lambda$-eigensections of $\mathbb D_{\Sigma_i}$. 

Set $\lambda_i:=\min\{\lambda\in \operatorname{spec} \mathbb D_{\Sigma_i}: \mu_i\leq \lambda\}$. If $\lambda_i\geq 2\mu_i$, then $\sigma_i\in \Omega^{1}_{-2\mu_i}(L_i)$, and we can iterate this argument to improve the decay rate $\mu_i$ of $\sigma_i$. But the decay rate of $\nu_i$ is bounded below by a fixed exponential rate, so this process must stop in finitely many steps, so without loss $\lambda_i<2\mu_i$. Thus
\[
\nabla^k ((\sigma_i,0)- e^{-\lambda_i t} \eta_{\lambda_i}) =o(e^{-\lambda_{i} t}),\ \forall k\geq 0 \quad \text{as $t\to +\infty$}.\]
The Lemma follows by comparing the $\Omega^1(L_i)$ component. 
\end{proof}

We proceed to prove the graphicality of $L_{\text{red}}$ near the open edges using the asymptotic behavior of $L$.

\begin{lemma}[Graphicality near the open edges]\label{Graphicality near the open edges} 
There is an open neighborhood $U_i$ of the open edge $(p_i,p_{i+1})$, such that $L_{\text{red}}$ is a smooth graph over $U_i$. 
\end{lemma}

\begin{proof} 
We examine the end modeled on $L_i$. As before, after rotation of the $u_1, u_2$ plane, we may assume that the open edge $(p_i,p_{i+1})$ lies on $\{ u_1=0\}$, the polygon lies on $\{u_1\geq 0\}$, and we express $(u_1, y_2)$ as functions of $(u_2, y_1)$ satisfying the elliptic system (\ref{eqn:ellipticsystem}). Recall that the asymptotic cylinder $L_i$ projects to $[p_i,p_{i+1}]\times \{y=(y_1,y_2)\in \mathbb R^2 \; | \; (p_{i+1}-p_i) \cdot y=c_i\}$. After translation by a constant in $(y_1,y_2)$, we may assume $c_i=0$.

The asymptotic formula  (\ref{eq asymp formula 1 of SL}) translates to the following estimate upon $U(1)$ dimensional reduction: for each $k\geq 0$,
\begin{equation*} \begin{cases}
\nabla^k\big(u_1-e^{-\lambda y_1}a(u_2)\big)=o(e^{-\lambda y_1}) \quad \text{as $y_1\to +\infty$},
\\
\nabla^k\big(y_2-e^{-\lambda y_1}b(u_2)\big)=o(e^{-\lambda y_1}) \quad \text{as $y_1\to +\infty$},
\end{cases}
\end{equation*}
 for some constant $\lambda>0$ and functions $a,b:(p_i,p_{i+1})\to \mathbb R$ that are not both identically zero. The above estimates are uniform for $u_2$ in any compact subset of $ (p_i,p_{i+1})$.
In particular,
\begin{equation*} 
 \frac{\partial u_1}{\partial u_2}= e^{-\lambda y_1}a^\prime(u_2)+o(e^{-\lambda y_1})=-\frac{\partial y_2}{\partial y_1}=\lambda e^{-\lambda y_1}b(u_2)+o(e^{-\lambda y_1})\quad \text{as $y_1\to +\infty$}.
\end{equation*}
Comparing the leading order terms, we obtain  $a^\prime(u_2)=\lambda b(u_2)$. Thus, $a$ can not be identically zero. Since $u_1\geq 0$, the function $a$ is non-negative.

The function $u_1$ satisfies the PDE (\ref{eq elliptic pde for max principle}), which can be rewritten using the the elliptic system (\ref{eqn:ellipticsystem}) as 
 \begin{equation}
 \label{eq rephrased elliptic pde for max principle}
         \frac{\partial^2 u_1}{\partial y_1^2} +\frac 1V \frac{\partial^2 u_1}{\partial u_2^2} + \frac 1 {V^3}\frac{\partial V}{\partial u_1 } \Big(\frac{\partial y_2}{\partial u_2}\Big)^2 = 0.
    \end{equation}
    Since $\tfrac{1}{V^3}\tfrac{\partial V}{\partial u_1 }=O(1)$ as $y_1\to +\infty$, substituting the above exponential asymptote of $u_1$ and $y_2$ into the PDE (\ref{eq rephrased elliptic pde for max principle}), we deduce from the $y_1\to +\infty$ leading order asymptote that
\begin{equation}\label{eq a eigen function}
      \lambda^2 a +\frac 1{V(0,u_2)} \frac{d^2 a}{du_2^2}=0,\quad u_2\in (p_i,p_{i+1}).
    \end{equation}
The strong maximum principle then implies that $a>0$ for all $u_2\in (p_i,p_{i+1})$.

In particular, given any compact subset of the open edge $(p_i,p_{i+1})$ in the polygon, then for $y_1$ large enough, we have
\begin{equation*} 
\frac{\partial u_1}{\partial y_1}=-\lambda e^{-\lambda y_1}a(u_2)+o(e^{-\lambda y_1}) <0.
\end{equation*}
By the implicit function theorem, the projection map from this subset of $L_{\text{red}}$ in the $L_i$ end region to the $(u_1,u_2)$ coordinates is then a diffeomorphism onto the image, which covers a small neighborhood of the given compact subset of $(p_i,p_{i+1})$.

On the other hand, by the asymptote along the other ends and the properness of the projection to $(u_1,u_2)$-coordinates, we deduce that only points in $L_i$ can project to  small enough neighborhoods of the compact subsets of $(p_i,p_{i+1})$. The Lemma follows. 
\end{proof}

By the properness of $\pi_u$, this implies that

\begin{corollary}\label{cor_deg 1}   The projection map has degree one on the interior region of the polygon.
\end{corollary}

From now on, we assume that $L$ is homeomorphic to $S^3$ minus $n$ points. {As shown in Step 2 of Lemma \ref{Lemma_manifold with boundary}, this implies that $L_{\text{red}}$ is simply connected.}

\begin{lemma}
\label{Lemma_submersion} Suppose $L$ is homeomorphic to an $n$-holed 3-sphere. Then, the projection $\pi_u$ is a local diffeomorphism in the interior of $L_{\text{red}}$.
\end{lemma}

\begin{proof} 
Suppose there is a point $z_0$ in the interior of $L_{\text{red}}$ where ${du_1\wedge du_2}=0$ that is, $d\pi_u$ is not surjective at $z_0$. Set $u_0:=\pi_u(z_0)$, a point in the interior of the polygon. Let $v$ be a vector in the $(u_1,u_2)$-plane such that image of $d\pi_u$ at $z_0$ is orthogonal to $v$. We define an affine linear function $f:L_{\text{red}}\to \mathbb R$ by 
$$f(z)=\langle\pi_u(z)-u_0,v\rangle .$$
Here $\langle\cdot, \cdot \rangle$ is the standard Euclidean inner product on $\mathbb R_u^2$. Note that $z_0\in f^{-1}(0)=\pi_u^{-1}(u_0+\ell)$, where $\ell:=v^\perp\subset \mathbb R^2_u$ and ${df}_{|z_0}=0$. We want to get a contradiction by studying the set $f^{-1}(0)$.

\textbf{Step 1} ($f^{-1}(0)$ near a singular interior point). By the implicit function theorem, the zero locus of $f$ is locally a smooth curve at any interior point of $L_{\text{red}}$ where $df
\neq 0$. Let $w$ be an interior point of $L_{\text{red}}$ lying in $f^{-1}(0)$ with ${df}= 0$ at $w$. After an affine transformation, we may assume that $f=u_1$. We claim that $u_1^{-1}(0)$ near $w$ is a union of at least four Jordan arcs emanating from $w$ and are disjoint after removing $w$. In particular, this holds at $w=z_0$. As in the proof of the Lemma \ref{lem-proj-to-polygon}, near $w$, we can regard $u_1$ and $y_2$ as functions of $y_1$ and $u_2$, satisfying the elliptic system (\ref{eqn:ellipticsystem}).
    Evidently, $u_1$ and $y_2$ are real analytic functions of $y_1$ and $u_2$ for being solutions of an elliptic PDE with real analytic coefficients. Set $w:=(0,a,b,c)\in \mathbb R^2_u\times \mathbb R^2_y$, so $(y_1,u_2)=(b,a)$ is a zero of multiplicity $m\geq 1$ of $(u_1,y_2-c)$. We consider the truncated Taylor polynomials
\[p(y_1,u_2):=\sum_{j=0}^m\frac{(y_1-b)^j (u_2-a)^{m-j}}{j! (m-j)!} \frac{\partial^m u_1}{\partial y^j_1\partial u^{m-j}_2}(b,a)\]
and 
\[q(y_1,u_2):=\sum_{j=0}^m\frac{(y_1-b)^j (u_2-a)^{m-j}}{j! (m-j)!}  \frac{\partial^m y_2}{\partial y^j_1\partial u^{m-j}_2}(b,a).\]
Then $u_1$ and $y_2$ can be expressed near $(b,a)$ as
\[u_1=p(y_1,u_2)+ O\big(\lvert y_1-b \rvert^{m+1}+\lvert u_2-a \rvert^{m+1}\big)\]
and
\[y_2=c+q(y_1,u_2)+ O\big(\lvert y_1-b \rvert^{m+1}+\lvert u_2-a \rvert^{m+1}\big).\]
  Since $V(0,a)\neq 0$ and $V(u_1,u_2)=V(0,a)+O(\lvert u_1 \rvert+\lvert u_2-a \rvert)$. Comparing the $m$-th order terms in the equation (\ref{eqn:ellipticsystem}) near $(b,a)$, we obtain that
\begin{align*}
        V(0,a) \frac{\partial p}{\partial y_1}  = \frac{\partial q}{\partial u_2}, 
        \quad \quad \text{and}  \quad \quad
        \frac{\partial p}{\partial u_2}  = -\frac{\partial q}{\partial y_1}.
    \end{align*}
This is the Cauchy-Riemann equation for $\sqrt{V(0,a)} p+iq$ as a function of $y_1+i \sqrt{V(0,a)}{u_2}$. Thus
\[\sqrt{V(0,a)} u_1+i(y_2-c)=c'\Big(y_1-b+i \sqrt{V(0,a)}{(u_2-a)}\Big)^m+O\big(\lvert y_1-b \rvert^{m+1}+\lvert u_2-a \rvert^{m+1}\big) \]
for some $ c'\in \mathbb C^*$. Since ${du_1}=0$ at $(b,a)$, hence $m\geq 2$. This proves our claim.

\textbf{Step 2} ($f^{-1}(0)$ does not contain any closed loop). This step uses the assumption that $L_{\text{red}}$ is simply connected. Indeed, if $f^{-1}(0)$ contains a closed loop, then simply connectedness implies that this loop bounds a disc in $L_{\text{red}}$. Hence, $f$ admits a minimum or maximum inside this disc, say at $p$, which is an interior point of $L_{\text{red}}$. But in the proof of the Lemma \ref{lem-proj-to-polygon}, we have seen that $f$ near $p$ satisfies the elliptic PDE (\ref{eq elliptic pde for max principle}), and the strong maximum principle yield that $f=0$ over the disc, which is a contradiction.

\textbf{Step 3}. Using the structure of $f^{-1}(0)$, we will now derive a contradiction.  Since $df=0$ at $z_0$, there are at least four paths in $f^{-1}(0)$. If along these paths,
we encounter an interior point of $L_{\text{red}}$ where $df=0$, then we choose a branch of $f^{-1}(0)$ to continue the path. These continued paths do not meet, as $f^{-1}(0)$ does not contain any closed loop. 

\textbf{Case 1.} Suppose the line $u_0+\ell$ does not pass through any of the vertices of the polygon, then all of the above paths can continue to infinity only along the two asymptotic ends corresponding to the two edges intersecting $u_0+l$. Lemma \ref{Graphicality near the open edges} implies that each such end can contain at most one path projecting to $u_0+l$.  Since there are at least four paths, this leads to a contradiction. 

\textbf{Case 2.} Suppose $u_0+\ell$ passes through one vertex $p_i$ and an open edge. In this case, again Lemma \ref{Graphicality near the open edges} implies that at most one of the above paths can continue to infinity along the cylindrical end corresponding to the edge. Since $u_0+\ell$ passes through $p_i$, we have that the boundary component $\pi_u^{-1}(p_i)\cap L_{\text{red}}\subset f^{-1}(0)$. The remaining paths (at least three) can have the following possibilities:

(i) At least two of them are meeting the boundary component $\pi_u^{-1}(p_i)\cap L_{\text{red}}$. This would imply that $f^{-1}(0)$ contains a closed loop formed by these two paths together with a portion of $L_{\text{fix}}\cap \pi_u^{-1}(p_i)$. This loop bounds a disc, and $f$ is constant on the boundary loop, so we again apply the strong maximum principle of Lemma \ref{lem-proj-to-polygon} to deduce a contradiction.

(ii) At least two of them are meeting the boundary component $\pi_u^{-1}(p_i)\cap L_{\text{red}}$ at infinity. In this case, we define a region $\mathcal R$ as follows. If these two paths go to infinity along one end then $\mathcal R$ is the region in between them, otherwise it is the region in between them and the above boundary component. Although the region $\mathcal R$ is unbounded, the fact that $f=0$ along its boundary and the limit of $f$ is zero at the infinity of $\mathcal{R}$, means that we can still run the strong maximum principle argument to deduce $f=0$ on $\mathcal{R}$, contradiction.

\textbf{Case 3.} Suppose $u_0+\ell$ passes through two distinct vertices $p_i$ and $p_j$.  Then again the Lemma (\ref{Graphicality near the open edges}) implies that at least four such paths  should meet the boundary components $\pi_u^{-1}(p_i)\cap L_{\text{red}}$ or $\pi_u^{-1}(p_j)\cap L_{\text{red}}$. At least two of them must meet one of the boundary components at some points or at infinity, and the same argument as Case 2 reaches a contradiction.
\end{proof}

We can now finish the proof of the main theorem, using the graphicality of $L_{\text{red}}$ over the interior of the polygon.

\begin{proof}[Proof of Theorem \ref{uniqueness}] 
The projection of $L_{free}/U(1)\subset L_{\text{red}}$ to the $(u_1,u_2)$-coordinates is a degree one proper covering map over the interior of the convex polygon, hence $L_{free}/U(1)$ is a \emph{smooth graph} over the interior of the polygon.

The condition $du_1 \wedge dy_1 + du_2 \wedge dy_2 = 0$ implies that $(y_1,y_2) = \nabla \varphi(u_1,u_2)$ for some function $\varphi$ on the interior of the convex polygon. The condition $V du_1 \wedge du_2 = dy_1 \wedge dy_2$ is equivalent to $\det D^2 \varphi = V$. Since the volume forms on $L$ is given by
    \begin{align*}
        Re(\Omega) = \theta \wedge (du_1 \wedge dy_2 - du_2 \wedge dy_1) > 0,
    \end{align*}
we deduce $tr D^2 \varphi = \iota_{\theta} Re(\Omega)(\partial u_1, \partial u_2) > 0$, so $\varphi$ must be a convex function.

We shall examine the boundary condition on $\varphi$. We claim that on any edge $[p_i,p_{i+1}]$ of the polygon, $\varphi$ extends continuously to an \emph{affine linear} function. As usual, without loss the open edge is contained in $\{ u_1=0\}$, and we can express $(u_1, y_2)$ as a function of $(u_2, y_1)$. This corresponds to taking the partial Legendre transform of the convex function $\varphi$:
\[
\phi= \varphi- u_1 y_1, \quad u_1=- \frac{\partial \phi}{\partial y_1},\quad y_2= \frac{\partial \phi}{\partial u_2}.
\]
By the asymptotically cylindrical condition along the end $L_i$,
\[
u_1= O(e^{-\mu_i y_1}), \quad (p_{i+1}-p_i)y_2 = c_i +O( e^{-\mu_i y_1}),\quad y_1\to +\infty. 
\]
The exponential decay of the special Lagrangian implies that its geometric distance to the cylindrical model decays exponentially. Consequently, the functions $u_1$ and $(p_{i+1}-p_i)y_2 - c_i$ also decay exponentially, and therefore, along each end $L_i$, the uniform convergence holds above the whole closed interval $[p_i, p_{i+1}]$.
Upon integration, as $y_1\to +\infty$, $\phi$ and hence $\varphi$ converge to some affine linear function in $u_2$, with convergence rate $O(e^{-\mu_i y_1})$. This argument shows that along the end $L_i$, $\varphi$ extends continuously to an affine linear function on $[p_i, p_{i+1}]$, which we may write as
\begin{equation}
\varphi( s p_i+ (1-s) p_{i+1}) = s b_i+ (1-s) b_{i+1},
\end{equation}
for real numbers $b_i, b_{i+1}$, such that $b_{i+1}-b_i=c_i$. The caveat here is that each vertex $p_i$ is shared by two edges $[p_i,p_{i+1}]$ and $[p_{i-1}, p_i]$, and we need to show that the two values assigned to $\varphi(p_i)$ are equal.

We view $\varphi$ as a function on the special Lagrangian $L\setminus L_{\text{fix}}$ satisfying $d\varphi= y_1 du_1+ y_2 du_2$, and extending continuously to a function on $L$. Recall from Step 2 of Lemma \ref{Lemma_manifold with boundary} that $L_{\text{fix}}\cap \pi_u^{-1}(p_i)$ is homeomorphic to $\mathbb{R}$, with two ends going off to infinity along the $L_i$ and $L_{i-1}$ ends respectively. Since $u_1$ and $u_2$ are moment maps for the $U(1)$-action on $X$, $du_1=du_2=0$ at the $U(1)$-fixed points. Therefore the function $\varphi$ must be locally constant on $L_{\text{fix}}$, so $\varphi$ is constant on the unique component of $L_{\text{fix}}$ lying above $p_i$, and in particular both ends $L_i$ and $L_{i-1}$ assign the same limiting value to $\varphi(p_i)$. Thus $\varphi$ 
extends continuously to the vertices $p_i$ of the polygon.

In conclusion, $\varphi$ is a continuous convex function on the polygon, solving the real Monge-Amp\`ere equation $\det(D^2\varphi)=V$ in the interior, and satisfies the affine linear boundary condition on the edges. This Dirichlet problem admits a unique solution, for instance as in \cite[Proposition 3.3]{figalli2018monge} and \cite[Theorem 6.3]{harvey2009dirichlet}, and the gradient graph of this solution agrees precisely with the special Lagrangian construction in \cite[Theorem 3]{esfahani2024donaldson}, so $L$ belongs to the family of special Lagrangians constructed there.
\end{proof}

\appendix

\section{Appendix: topology}\label{topology}

In this appendix, we prove that any special Lagrangian \( L \subset X \times \mathbb{C} \) satisfying the asymptotic condition of the generalized local Donaldson-Scaduto conjecture is homeomorphic to either an \( n \)-holed 3-sphere or an \( n \)-holed connected sum of \( S^1 \times S^2 \). 

\begin{theorem}
\label{Thm_PossibleTopology}
    Let $L \subset X \times \mathbb{R}^2$ be a special Lagrangian submanifold with $n$ asymptotically cylindrical ends modeled on the product of $\Sigma_i$ and $\{y \in \mathbb{R}^2 \; | \; (p_{i+1} - p_i) \cdot y = c_i \}$. Then {$L$ is $U(1)$-invariant and $L/U(1)$ is a surface with boundary, and}
    \begin{align*}
        L \cong S^3 \setminus \{ \text{$n$ points} \} \quad \text{ or } \quad L \cong \#_{2g+b} (S^2 \times S^1) \setminus \{ \text{$n$ points} \}, \;
    \end{align*}
where $g$ is the genus and $b$ is the number of circle boundary components of the surface $L/U(1)$.
\end{theorem}
\begin{remark}\emph{ The Donaldson-Scaduto conjecture in \cite{donaldson2020associative} predicts the uniqueness of the special Lagrangian among all 3-dimensional special Lagrangian submanifolds satisfying given asymptotic conditions, not just special Lagrangian pairs of pants. Therefore, we expect that the second alternative for the topology is not realized by special Lagrangians with the prescribed cylindrical asymptote.}
\end{remark}

The proof relies on the following theorem of Raymond.

\begin{theorem}[{{\cite[Raymond, Theorem 1(i)-(ii)(a)]{raymond1968classification}}}]
\label{Thm_Raymond}
Every closed, orientable 3-manifold that admits an effective $U(1)$-action with fixed points and no non-trivial finite isotropy subgroups is homeomorphic to
\[S^3\#(S^2\times S^1)_1\#\dots \#(S^2\times S^1)_{2g+h-1},\]
where $g$ is the genus of the orbit space and $h$ is the number of connected components of the fixed point set.
\end{theorem}

\begin{proof}[Proof of Theorem \ref{Thm_PossibleTopology}]
{$L$ is $U(1)$-invariant and $L/U(1)$ is a surface with boundary homeomorphic to $L_{\text{fix}}$, as in Step 1 of Lemma \ref{Lemma_manifold with boundary}.}

To prove the homeomorphism type we compactify $L$ as follows. Let $\tilde{L}$ be the closed 3-manifold obtained by capping off the $n$ cylindrical ends with 3-balls. Note that at large distance, the $U(1)$-action is smoothly equivalent to the standard $U(1)$-action on the $S^2$-fibers at the ends of $L$, and it preserves these 2-sphere cross-sections. Specifically, the standard $U(1)$-action on $S^2 \subset \mathbb{R}^3$ as rotation around a fixed axis extends to the 3-ball by acting on each concentric $S^2$ centered at the origin via rotation around the same axis. This gives a smooth extension of the $U(1)$-action to $\tilde{L}$.

Since $L$ is a special Lagrangian, it is orientable, so $\tilde{L}$ is also orientable. {The $U(1)$-action on $\tilde L$ is obviously effective as it is on $L$, and the fixed point set  is non-empty. Since the $U(1)$-action on $X$ had only isotropy subgroups $U(1)$ and $\{1\}$, so is the action on $\tilde L$, namely there are no non-trivial finite isotropy subgroups. Thus by Theorem \ref{Thm_Raymond},
$$\tilde L\cong S^3\#(S^2\times S^1)_1\#\dots \#(S^2\times S^1)_{2g+h-1}.$$
where $g=\text{genus}( \tilde L/U(1))= \text{genus}(L/U(1))$, and $h$ is the number of connected components of the fixed point set  $\tilde L_{\text{fix}}$. 

In Step 2 of Lemma \ref{Lemma_manifold with boundary}, {we have seen that $L_{\text{fix}}$ has $n$ connected components that are homeomorphic to $\mathbb R$, extending from one end $L_i$ to the other end $L_{i+1}$}. Consequently, through the above capping off procedure, these combine to form one single connected component of  $\tilde L_{\text{fix}}$. Hence $h=1+b$, where $b$ is the number of circle components of $L_{\text{fix}}$. The claim follows by removing $n$ disjoint balls from $\tilde{L}$.}
\end{proof}

\printbibliography

\noindent
\author{Simons Center for Geometry and Physics, State University of New York, Stony Brook, NY 11794} \\ E-mail address: \href{ mailto:gbera@scgp.stonybrook.edu}{gbera@scgp.stonybrook.edu}

\vspace{10pt}

\noindent
\author{Department of Mathematics, Duke University, 120 Science Dr, Durham, NC 27708-0320} \\ E-mail address: \href{ mailto:saman.habibiesfahani@duke.edu}{saman.habibiesfahani@duke.edu}

\vspace{10pt}

\noindent
\author{Department of Pure Mathematics and Mathematical Statistics, Cambridge University, Wilberforce Road, Cambridge CB3 0WA, UK
} \\ E-mail address: \href{ mailto:yl454@cam.ac.uk}{yl454@cam.ac.uk}

\end{document}